\documentclass{amsart}

\usepackage{amsthm}
\usepackage{amssymb}
\usepackage{latexsym}
\usepackage{euscript}
\usepackage{amsmath}
\usepackage{amstext}
\usepackage{amsgen}
\usepackage{amsbsy}
\usepackage{amsopn}
\usepackage{amsfonts}
\usepackage{hyperref} 
\usepackage{graphicx}
\usepackage{bbm}
\usepackage{tikz}
\usepackage{tikz-cd}    
\usepackage{enumerate}
\usepackage{color}
\usepackage{a4wide}
\usepackage{mathrsfs}
\usepackage{float}
\usepackage{array}
\newcolumntype{M}[1]{>{\centering\arraybackslash}m{#1}}
\usepackage{relsize}
\usepackage{mathtools}
\usepackage{adjustbox}

\usepackage{thmtools} %added so that cleveref correctly identifies results with same counter
\usepackage[capitalize]{cleveref}
\crefname{theorem}{Theorem}{Theorems}
\Crefname{theorem}{Theorem}{Theorems}
\crefname{lemma}{Lemma}{Lemmas}
\Crefname{lemma}{Lemma}{Lemmas}
\crefname{proposition}{Proposition}{Propositions}
\Crefname{proposition}{Proposition}{Propositions}
\crefname{corollary}{Corollary}{Corollaries}
\Crefname{corollary}{Corollary}{Corollaries}
\crefname{definition}{Definition}{Definitions}
\Crefname{definition}{Definition}{Definitions}
\crefname{example}{Example}{Examples}
\Crefname{example}{Example}{Examples}
\crefname{remark}{Remark}{Remarks}
\Crefname{remark}{Remark}{Remarks}
\crefname{equation}{}{}
\Crefname{equation}{}{}

\usetikzlibrary{matrix,arrows,decorations.pathmorphing}

\theoremstyle{definition}

\newtheorem{theorem}{Theorem}[section]
\newtheorem{proposition}[theorem]{Proposition}
\newtheorem{corollary}[theorem]{Corollary}
\newtheorem{lemma}[theorem]{Lemma}
\newtheorem{thmx}{Theorem}

\theoremstyle{definition}
\newtheorem{definition}[theorem]{Definition}
\newtheorem{example}[theorem]{Example}
\newtheorem{remark}[theorem]{Remark}

\newcommand{\defn}[1]{\emph{#1}}
\newcommand{\N}{\mathbb{N}}
\newcommand{\Z}{\mathbb{Z}}

\newcommand{\R}{\mathbb{R}}
\newcommand{\C}{\mathbb{C}}

\newcommand{\kk}{\mathbbm{k}}
\newcommand{\PP}{\mathbb{P}}
\newcommand{\im}{\mathrm{im}}
\newcommand{\cat}[1]{\mathrm{#1}}

\newcommand{\coker}{\mathrm{coker}\,}
\newcommand{\kvect}{\cat{Vect}_{\kk}}
\newcommand{\loc}[1]{\cat{Loc}\!\left( #1\right)}
\newcommand{\ep}[1]{\cat{EP}\!\left( #1\right)}
\newcommand{\perv}[2]{{}^{#1}\cat{Perv}\!\left(#2\right)}
\newcommand{\pfun}[2]{{}^{#1}{#2}}

\newcommand{\der}[2]{\mathrm{D}^{#1}({#2})}
\newcommand{\shf}[1]{\mathrm{Sh}\!\left( #1\right)}
\newcommand{\constr}[2]{\cat{D}_{#1}({#2})}

\newcommand{\sh}[1]{\mathcal{#1}}
\newcommand{\id}{\mathrm{id}}
\newcommand{\extf}[2]{\mathbb{E}_{#1}(#2)}
\newcommand{\sextf}[2]{\mathbb{SE}_{#1}(#2)}
\newcommand{\unc}[1]{\Psi_f^{\mathrm{un}}(#1)}
\newcommand{\uncc}{\Psi_f^{\mathrm{un}}}
\newcommand{\nuncc}{\Psi_f^{\neq1}}

\newcommand{\Mor}[3]{{\mathrm{Hom}}_{#1}\!\left(#2,#3\right)}

\newcommand{\Ext}[4]{{\mathrm{Ext}_{#1}^{#2}}\!\left(#3,#4\right)}

\newcommand{\dual}{\mathbb{D}}
\newcommand{\pt}{{\bf pt}}

\newcommand{\monic}{\hookrightarrow}

\hypersetup{
    colorlinks=true,
    linkcolor=red}

\makeatletter
\@namedef{subjclassname@2020}{\textup{2020} Mathematics Subject Classification}
\makeatother

\newcommand{\pcoh}[2]{{{^p}H^{#1}(#2)}}

\newcommand{\Plsh}{\mathrm{P}_{!}}
\newcommand{\Pls}{\mathrm{P}_{\!*}}

\title{Indecomposable extensions of perverse sheaves over a closed stratum}  
\author{Alessio Cipriani}
\address{Alessio Cipriani, Dipartimento di Informatica - Settore di Matematica, Universit\`a degli Studi di Verona, Strada le Grazie 15 - Ca' Vignal, I-37134 Verona, Italy}
\email{alessio.cipriani@univr.it}
\begin{document} 

\begin{abstract}
Given a topologically stratified space $X$, we develop a categorical framework for extensions of perverse sheaves over a closed stratum $S$. We introduce the notion of $S$-small extensions and extension pairs relative to a fixed perverse sheaf on $X\smallsetminus S$. By replacing the homotopical assumptions $\pi_1(S)=\pi_2(S)=0$ in the work of MacPherson and Vilonen \cite{MR833195} with the categorical condition that the category of local systems on $S$ is semisimple, we construct an equivalence of additive categories between $S$-small extensions and extension pairs. This extends the MacPherson--Vilonen description to a more general setting and yields a maximal extension functor, generalising Beilinson's construction. As a consequence, we obtain a structural classification of indecomposable $p$-perverse sheaves on $X$: every such object is either an extension by zero of an indecomposable local system on $S$, or arises from an extension pair whose canonical morphism does not decompose as a direct sum.
\end{abstract} 


\maketitle
\section{Introduction}

Categories of perverse sheaves on a possibly singular space $X$ were introduced by Beilinson, Bernstein and Deligne in \cite{BBD}. They arise as the heart of a t-structure, determined by a $\Z$-valued function $p$ on $X$ called perversity, on the constructible derived category of $X$. When $X$ is a topologically stratified space with finitely many strata, the category $\perv{p}{X}$ of $p$-perverse sheaves on $X$ is Hom-finite, abelian and length for any choice of the perversity. By \cite[\S5]{Krause}, it follows that $\perv{p}{X}$ is Krull--Schmidt, so that every $p$-perverse sheaf decomposes as a finite direct sum of indecomposable objects. Under the further assumption that the strata of $X$ have finite fundamental group, the category $\perv{p}{X}$ is equivalent to a category of finite-dimensional modules over a finite-dimensional $\kk$-algebra, independently of the choice of the perversity function, see \cite[Corollary 5.3]{CIPRIANI_2021}. The Krull--Schmidt property of $\perv{p}{X}$ motivates the study of indecomposable perverse sheaves, that is objects with local endomorphism algebra. However, the problem of classifying all indecomposable perverse sheaves is in general rather hard, as extension of decomposable perverse sheaves can give rise to indecomposable ones, see \cref{example:kronecker}. 

Our goal is to understand extensions of perverse sheaves on $X\smallsetminus S$ over a closed stratum $S\subset X$, where $X$ is a topologically stratified space. This strategy is the same used in \cite{MR833195} and reduces the problem to an inductive argument on the stratification since the strata of $X$ can be ordered so that each stratum is closed in the union of itself with the preceding ones. To avoid the machinery of derived categories, MacPherson and Vilonen give two versions of how one can uniquely extend a perverse sheaf on the open complement of a closed stratum to a perverse sheaf on $X$ under the assumption that $\pi_1(S)=\pi_2(S)=0$, see \cite[Proposition 6.1]{MR833195}. The first one is for topologically stratified spaces, see \cite[Theorem 4.5]{MR833195}, and the second for complex analytic manifolds, see \cite[Theorem 5.3]{MR833195}. In both cases, they obtain an equivalence between the category of middle perverse sheaves and an auxiliary abelian category whose objects are a perverse sheaf on $X\smallsetminus S$ and some linear-algebraic data encoded in a commutative diagram of sheaves of vector spaces. They also present versions of the inductive step which work under extra information on the structure normal to the closed stratum $S$. They work with a specific choice of perversity on $X$, namely the middle perversity, although they remark that their results extend to arbitrary perversities.

In this paper, we work in the general setting of a topologically stratified space, and we allow $p$ to be any perversity function on $X$. We do not require any restriction on the number of strata. In particular, we do not assume that $\perv{p}{X}$ is equivalent to a category of modules over a finite-dimensional algebra or that it is satisfies the Krull--Schmidt property. Our only structural assumption is that the category of local systems on the closed stratum $S\subset X$ is semisimple, equivalently that all the $\mathrm{Ext}^1$-groups between local systems on $S$ vanish. This assumption is satisfied, for instance, when $S$ has finite fundamental group and one works with a field whose characteristic does not dived its order. In contrast with \cite{MR833195}, we do not make any assumption on $\pi_2(S)$, however when all the (triangulated) $\mathrm{Ext}^2$-groups vanish, which happens if for instance $H^2(S;\sh{L})\cong0$ for any $\sh{L}\in\loc{S}$, then our construction becomes simpler. Under these minimal assumptions, we develop a framework for extending perverse sheaves over $S$ that differs substantially from that of \cite{MR833195}, and that applies in a strictly more general setting.

We review the necessary background material in \cref{section:background} and in \cref{extf} we introduce the category $\extf{S}{\sh{F}}$ of extensions of a fixed perverse sheaf $\sh{F}\in\perv{p}{X\smallsetminus S}$ over the closed stratum $S\subset X$. \cref{section:ProjfuncPlshPls} is devoted to the introduction of the projection functors $\Plsh,\Pls\colon\perv{p}{X}\to\perv{p}{X}$ and to the study how properties of morphisms behave under them. We note that for any object $\sh{E}\in\extf{S}{\sh{F}}$ there is a canonical morphism $\Plsh\sh{E}\to\Pls\sh{E}$, see \cref{can morphism proj functors}, given by the factorisation through the intermediate extension of $\sh{F}\in\perv{p}{X\smallsetminus S}$. We define the notion of $S$-small perverse sheaves in \cref{subsec_small_exts}, that is objects $\sh{E}\in\perv{p}{X}$ in the kernel of the intermediate restriction functor $\pfun{p}{{\imath}^{!*}}\coloneqq\im(\pfun{p}{\imath^!}\to\pfun{p}{\imath^*})$, see \cref{restriction_functors} and we introduce the category $\sextf{S}{\sh{F}}$ of $S$-small extensions. In \cref{small_no_summ_S} we characterise the $S$-small perverse sheaves as the objects with no summands on the closed stratum $S$ and we show that the category $\sextf{S}{\sh{F}}$ is additive, closed under quotients and sub-objects, but in general not abelian, see \cref{rem:SESF_not_abelian}. In \cref{subsection:extending_pervsheaves}, we introduce extension pairs relative to a fixed perverse sheaf $\sh{F}$ on $X\smallsetminus S$, see \cref{ext_pair_relative}, that is pairs of objects in $\perv{p}{X}$ which are extensions of $\sh{F}$, the first with no quotient and the second with no sub-object on the closed stratum $S\subset X$ respectively, and such that a certain morphism in the constructible derived category vanishes. Our first main result establishes a bijection between isomorphism classes of extension pairs and $S$-small extensions.

\begin{thmx}(\cref{thm:thmonetoonecorrsmallext})\label{thmA}
Let $X$ be a topologically stratified space, $S\subset X$ a closed stratum so that $\loc{S}$ is semisimple and $\sh{F}\in\perv{p}{X\smallsetminus S}$. 

There is a bijection between (isomorphism classes of) extension pairs $(\sh{A},\sh{B})$ relative to a fixed $\sh{F}\in\perv{p}{X\smallsetminus S}$ and (isomorphism classes of) $S$-small extension $\sh{E}(\sh{A},\sh{B})\in\extf{S}{\sh{F}}$.
\end{thmx} 

We prove \cref{thmA} by constructing a $S$-small extension relative to $\sh{F}\in\perv{p}{X}$ by applying the octahedral axiom to triangles involved in the definition of extension pair. Vice-versa, we show that starting with an $S$-small extension, the pair $(\Plsh\sh{E},\Pls\sh{E})$ is an extension pair since a certain (triangulated) $\mathrm{Ext}^2$-class obtained as product of two abelian $\mathrm{Ext}^1$-classes vanishes.

We note that \cref{thmA} extends beyond the situation of categories of perverse sheaves to the case of an abelian recollement $(\cat{A}_S,\cat{A},\cat{A}_{U})$, see \cite{BBD}, in which the Serre subcategory $\cat{A}_S\subset\cat{A}$ is semisimple, but we do not investigate this further.

\cref{thmA} also allows us to define the notion of the maximal extension of a fixed object $\sh{F}\in\perv{p}{X\smallsetminus S}$, see \cref{defn:generalised_max_ext}, as the (isomorphism class) of a perverse sheaf $\pfun{p}{\Xi}\sh{F}$ corresponding to a certain extension pair. This generalises the notion of Beilinson's maximal extension, see \cref{Beilinson_max_ext}, to a wider class of stratified spaces.  

In \cref{cat_of_ext pairs}, we introduce the category $\ep{X\smallsetminus S}$ of extension pairs over $S$ and we show that the result of \cref{thmA} can be upgraded to an equivalence of additive categories. In fact, since the constructible derived category of $X$ is an algebraic triangulated category, see \cref{constructible_der_cat}, the argument of \cref{thmA} is functorial.

\begin{thmx}(\cref{thm:equivalence_of_cat})\label{thmB}
Let $X$ be a topologically stratified space, $S\subset X$ be a closed stratum such that $\loc{S}$ is semisimple and $\sh{F}\in\perv{p}{X\smallsetminus S}$.

The assignment 
\begin{equation*}\begin{split}
\Theta \colon    \sextf{S}{\sh{F}} &\to \ep{X\smallsetminus S}\\ 
  \sh{E} & \mapsto (\Plsh\sh{E},\Pls\sh{E})
\end{split}\end{equation*}
defines an equivalence of additive categories.
\end{thmx}

As consequence of \cref{thmB}, we characterise the $S$-indecomposable perverse sheaves as the ones corresponding to extension pairs whose canonical morphism cannot be written as a direct sum, see \cref{cor_on_indecomp}.

We then use \cref{thmB} to characterise all indecomposable perverse sheaves $\sh{E}\in\perv{p}{X}$.

\begin{thmx}(\cref{thm:on:indec_exts})\label{thmC}
Let $X$ be a topologically stratified space, $S\subset X$ a closed stratum such that the category $\loc{S}$ is semisimple and $\sh{F}\in\perv{p}{X\smallsetminus S}$.

Isomorphism classes of indecomposable $p$-perverse sheaves on $X$ are classified by extensions by zero of indecomposable local systems on $S$ and extension pairs relative to $\sh{F}\in\perv{p}{X\smallsetminus S}$ whose canonical morphism cannot be written a direct sum.
\end{thmx}

\cref{thmC} can be seen as a reduction result, as it implies that any indecomposable perverse sheaf is either of the form $\imath_*\sh{L}[-p(S)]$ for an indecomposable $\sh{L}\in\loc{S}$, or can be described in terms of a pair of perverse sheaves easier to understand, namely an object with no quotients and one with no sub-objects on $S$ respectively.  

We point out that although we do not require $\sh{F}\in\perv{p}{X\smallsetminus S}$ to be indecomposable, the non-splitting of the canonical morphism $\Plsh\sh{E}\to \Pls\sh{E}$ is guaranteed if $\sh{F}\in\perv{p}{X\smallsetminus S}$ is an indecomposable object, see \cref{rmk_non_splitting}. 

 In \cref{sec_on_examples} we discuss examples of spaces with two strata.

\subsection*{Acknowledgements} The author is grateful to Jon Woolf for his feedback on different drafts of this manuscript and for introducing him to the ``ice cream cones" featured in \cref{fig:M1}. The author was funded by MUR PNRR-Seal of Excellence, CUP B37G22000800006, by the project funded by NextGenerationEU under NRRP, Call PRIN 2022 No. 104 of February 2, 2022 of Italian Ministry of University and Research; Project 2022S97PMY Structures for Quivers, Algebras and Representations (SQUARE), and by the project LAVIE - Large views of small phenomena: decompositions, localizations and representation type - Fondo Italiano per la Scienza FIS-2021 funded by of Italian Ministry of University and Research. A.C is member of the ``National Group for Algebraic and Geometric Structures, and their Applications" (GNSAGA - INdAM).

\section{Background}\label{section:background}

In this section we recall the necessary background material needed in the rest of the paper.

\subsection{Topologically stratified spaces}
Throughout the paper, $X$ will denote a {\em topologically stratified space}, as introduced in \cite{GMIH1,GMIH2}. By dimension of a topological manifold we always mean real dimension, unless otherwise specified.

A {\em 0-dimensional topologically stratified space} is a discrete union of points with the discrete topology. Let $d\in\N$, a {\em $d$-dimensional topologically stratified space} $X$ is a paracompact Hausdorff topological space endowed with a finite filtration by closed subsets of the form
\begin{equation*}
X=X_d\supset X_{d-1}\supset\ldots X_0\supset X_{-1}=\emptyset
\end{equation*}
such that:
\begin{enumerate}[(i)]
    \item every $X_i\smallsetminus X_{i-1}$ is a (possibly empty) topological $i$-manifold; the connected components of $X_i\smallsetminus X_{i-1}$ are the $i$-dimensional {\em strata} of $X$,
    \item for every $x\in X_i\smallsetminus X_{i-1}$, there exist an open neighbourhood $U\subset X$ and a compact $(d-i-1)$-dimensional topologically stratified space $L$, {\em the link of the stratum} at $x$, with a filtration preserving homeomorphism $\phi\colon U\to\R^i\times C(L)$, where $C(L)=L \times [0, 1)/L \times \{0\}$ denotes the open cone on $L$.
 \end{enumerate}

\subsection{Constructible derived category}\label{constructible_der_cat}  Let $\kk$ be a field and $X$ a topologically stratified space. We denote by $\shf{X}$ the abelian category of sheaves of $\kk$-vector spaces on $X$. The {\em constructible derived category of $X$}, denoted by $\constr{c}{X}$, is the full triangulated subcategory of the bounded derived category $\der{b}{X}=\der{b}{{\shf{X}}}$ consisting of constructible complexes. That is, the objects of $\constr{c}{X}$ are complexes of $\kk$-vector spaces on $X$ with  locally-constant  cohomology sheaves on each stratum $S$ of $X$. Note that the constructible derived category is algebraic, that is it admits a dg-enhancement by \cite[Theorem A]{Canonaco_2022} and \cite[Remark 3.2]{Canonaco_2017}. This implies that taking cones in $\constr{c}{X}$ is functorial.

Let $\imath\colon Z\to X$ be the inclusion of a closed union of strata and $\jmath\colon U=X\smallsetminus Z\to X$ be the complementary open inclusion. There is a triangulated recollement, see \cite[\S1.4.1]{BBD}
\begin{equation}\begin{tikzcd}\label{equation:6ff}
\constr{c}{Z}\ar[rr,"\imath_*","\perp"']&&\constr{c}{X}\ar[ll,"\imath^!",bend left=30]\ar[ll,"\imath^*","\perp"', swap,bend right=40]\ar[rr,"\jmath^*","\perp"']&&\constr{c}{U}\ar[ll,"\jmath_*",bend left=30]\ar[ll,"\jmath_!","\perp"', swap,bend right=40]\end{tikzcd}.\end{equation}
where $\imath_*$ and $\jmath^*$ are exact functors. By abuse of notation, we will denote a functor and its left (respectively right) derived functor with the same letter as already done in \cref{equation:6ff}.

Let us consider the map to a point $\pi\colon X\to \pt$ and let us denote by $\kk_{\pt}$ the skyscraper sheaf with stalk concentrated in degree zero. The triangulated equivalence defined by $\dual(-)=\Mor{\constr{c}{X}}{-}{\pi^!\kk_{\pt}}\colon\constr{c}{X}^{\mathrm{op}}\to\constr{c}{X}$ is called {\em Verdier duality}. The functor $\dual$ commutes with the exact functors $\imath_*$ and $\jmath^*$ of \cref{equation:6ff}. Moreover, there are natural isomorphisms $\dual \imath^*=\imath^! \dual$ and $\dual \jmath_!=\jmath_* \dual$.

\subsection{Perverse t-structures and perverse functors}\label{perverse_def_funct} Let $\imath_S\colon S\to X$ be the inclusion of a stratum into $X$ and $p$ a {\em perversity on $X$}, that is a function $p\colon \{\hbox{strata of}\ X\} \to\Z$. The pair of subcategories
\begin{equation*}
\begin{split}
    {^p{D^{\leq0}}}&=\{\sh{A}\in\constr{c}{X}\mid \sh{H}^k(\imath_S^*\sh{A})=0 \quad k>p(S), \ \forall S\subset X \} \\
    {^p{D^{\geq0}}}&=\{\sh{A}\in\constr{c}{X}\mid \sh{H}^k(\imath_S^!\sh{A})=0 \quad k<p(S), \ \forall S\subset X \}
\end{split}
\end{equation*}
is a bounded t-structure on $\constr{c}{X}$, see \cite[2.1.4]{BBD}, called the $p$-perverse t-structure, where $\sh{H}^k(\sh{E})$ denotes the cohomology sheaf of the complex of sheaves $\sh{E}\in\constr{c}{X}$. 

The category of $p$-perverse sheaves on $X$ is defined as the {\em heart} of such t-structure, that is $\perv{p}{X}={^p{D^{\leq0}}}\cap{^p{D^{\geq0}}}$, hence it is an abelian subcategory of $\constr{c}{X}$. 

A perverse  heart $\perv{p}{X}\subset \constr{c}{X}$ is {\em faithful} if there is a realisation functor $\der{b}{\perv{p}{X}}\to\constr{c}{X}$ which is an equivalence

The triangulated recollement \cref{equation:6ff} descends to an abelian recollement
\begin{equation}\begin{tikzcd}\label{equation:6ffab}
\perv{p}{Z}\ar[rr,"\imath_*","\perp"']&&\perv{p}{X}\ar[ll,"^p\imath^!",bend left=30]\ar[ll,"^p\imath^*","\perp"', swap,bend right=40]\ar[rr,"\jmath^*","\perp"']&&\perv{p}{U}\ar[ll,"\jmath_*",bend left=30]\ar[ll,"^p\jmath_!","\perp"', swap,bend right=40]\end{tikzcd}.
\end{equation}

The functors which appear in \cref{equation:6ffab} are called {\em perverse functors} and are defined by setting $^pF={^pH^0} F \epsilon$, where $F\in\{\imath^*,\imath_*,\imath^!,\jmath_!,\jmath^*,\jmath_* \}$ and $^pH^0\colon\constr{c}{Y}\to\perv{p}{Y}$ for $Y\in\{Z,X,U\}$ is the cohomological functor left inverse to the inclusion $\epsilon\colon\perv{p}{Y}\to\constr{c}{Y}$ of the perverse heart into the constructible derived category, see \cite[\S1.3.6]{BBD}.

Verdier duality restricts to an exact equivalence $\dual\colon\perv{p}{X}^{\mathrm{op}}\to\perv{p^*}{X}$, where $p^*$ is the {\em dual perversity} defined as $p^*(S)=-\dim_{\mathbb R}(S)-\dim(S)$, for any stratum $S\subset X$, see \cite[\S2.1.16]{BBD}.

\subsubsection{Extension functors.}\label{extension_functors}  In the rest of the paper we will say that $\sh{E}\in\constr{c}{X}$ is an {\em extension of} $\sh{F}\in\constr{c}{U}$ if $\jmath^*\sh{E}\cong\sh{F}$. We recall some well-known facts on extensions of perverse sheaves from $U$ to $X$. 

Let $\sh{F}\in\perv{p}{U}$, the object $\pfun{p}{\jmath_!}\sh{F}\in\perv{p}{X}$ is unique extension $\sh{E}$ of $\sh{F}$ in $\constr{c}{X}$ such that $\imath^*\sh{E}\in{^p{D^{\leq-2}}(Z)}$ and $\imath^!\sh{E}\in{^p{D^{\geq0}}(Z)}$, see \cite[Proposition 1.4.23]{BBD}. Similarly, $\pfun{p}{\jmath_*}\sh{F}$ is the unique extension $\sh{E'}$ of $\sh{F}$ in $\constr{c}{X}$ such that $\imath^*\sh{E'}\in{^p{D^{\leq0}}(Z)}$ and $\imath^!\sh{E'}\in{^p{D^{\geq2}}(Z)}$.  

The {\em intermediate extension functor} $\pfun{p}{\jmath_{!*}}\colon\perv{p}{U}\to\perv{p}{X}$ is defined, see \cite[Definition 1.4.22]{BBD}, as the image of the natural morphism $\pfun{p}{\jmath_!}\to \pfun{p}{\jmath_*}$, that is $\pfun{p}{\jmath_{!*}}:=\im(\pfun{p}{\jmath_{!}}\to \pfun{p}{\jmath_{*}})$. Let $\sh{F}\in\perv{p}{U}$, the intermediate extension $\pfun{p}{\jmath_{!*}}\sh{F}\in\perv{p}{X}$ of $\sh{F}$ is the unique extension $\sh{E''}$ of $\sh{F}$ in $\constr{c}{X}$ such that $\imath^*\sh{E''}\in{^p{D^{\leq-1}}(Z)}$ and $\imath^!\sh{E}\in{^p{D^{\geq1}}(Z)}$, see \cite[Corollaire 1.4.24]{BBD}. By \cite[Corollaire 1.4.25]{BBD}, $\pfun{p}{\jmath_{!*}}\sh{F}\in\perv{p}{X}$ of $\sh{F}\in\perv{p}{U}$ is also characterised as the unique extension of $\sh{F}$ with no non-zero sub-object and quotient supported on $Z$. In general, the intermediate extension functor is not exact.

\subsubsection{Restriction functors.}\label{restriction_functors} Let $Z$ be a closed union of strata. The next well-known result gives an algebraic interpretation of the functors in \cref{equation:6ffab} which are left and right adjoints to the closed inclusion $\imath_*$.

\begin{lemma}\label{max_sub_quot_on_Z}
Let $X$ be a topologically stratified space and $p$ a perversity on it. Let $\imath\colon Z\to X$ be the inclusion of a closed union of strata into $X$ and consider $\sh{E}\in\perv{p}{X}$. Then, $\imath_*\pfun{p}{\imath^!}\sh{E}\to \sh{E}$ is the inclusion of the maximal sub-object of $\sh{E}$ supported on $Z$ and $\sh{E}\to\imath_*\pfun{p}{\imath^*}\sh{E}$ is the projection on the maximal quotient of $\sh{E}$ supported on $Z$.
\end{lemma}

Similarly to how one defines the intermediate extension functor, we define the {\em intermediate restriction functor} $\pfun{p}{\imath^{!*}}\colon\perv{p}{X}\to\perv{p}{S}$ as the image of the natural morphism between the maximal sub-object and maximal quotient supported on $Z$, that is $\pfun{p}{\imath^{!*}}:=\im(\pfun{p}{\imath^!}\to\pfun{p}{\imath^*})$. The intermediate restriction functor is not exact in general.

\subsubsection{Beilinson's maximal extension.}\label{Beilinson_max_ext} We recall the definition, due to Beilinson, of yet another important extension functor - see \cite{MR923134} and \cite{MR2671769} for more details. 

In this section, we let $X$ be an algebraic variety and $m$ the middle perversity on it, defined by $m(S)=-\dim_{\mathbb C}S$ for any stratum $S\subset X$. Let $f\colon X\to \C$ be a regular function, denote by $X_0=f^{-1}(0)$ and $U=X\smallsetminus X_0$. Consider the diagram
\begin{equation*}
    \begin{tikzcd}
      X_0 \arrow{d}\arrow{r}{\imath} & X\arrow{d}{f} & U\arrow{d}\arrow{l}[swap]{\jmath} & \widetilde{U}\arrow{d}\arrow{l}[swap]{v}\\
      \{0\}\arrow{r} & \C & \C^*\arrow{l} & \widetilde{\C^*}\arrow{l}[swap]{u}
    \end{tikzcd}.
\end{equation*}
Then, the {\em nearby cycle functor} is defined as $\Psi_f:=\imath^*\jmath_*v_*v^*\colon \constr{c}{U}\to\constr{c}{X_0}$.

For any $\sh{G}\in\constr{c}{X}$, the fundamental group $\pi_1(\C^*)$ acts on $v^*\sh{G}$ via deck transformations of $\widetilde{\C^*}$ and hence it acts on $\Psi_f$. By \cite[Lemma 1.1]{MR2671769}, there is a unique decomposition of $\Psi_f$ as $\uncc\oplus\nuncc$ where any generator $t$ of $\pi_1(\C^*)$ is such that $1-t$ acts nilpotently on $\uncc(\sh{G})$ for any $\sh{G}\in\constr{c}{X}$ and is an automorphism of $\nuncc$. One usually refers to $\uncc$ as the {\em unipotent nearby cycle functor}.

Let $\sh{F}\in\perv{m}{U}$, {\em Beilinson's maximal extension} $\pfun{m}{\Xi}\sh{F}\in\perv{m}{X}$ is defined as the unique (up to isomorphism) perverse sheaf on $X$ for which there are two short exact sequences, see \cite[Proposition 3.1]{MR2671769},
\begin{equation*}
        0\to\jmath_!\sh{F}\xrightarrow{\alpha_-}\pfun{m}{\Xi}\sh{F}\xrightarrow{\beta_-} \imath_*\unc{\sh{F}}\to 0 \quad \hbox{and} \quad
        0\to\imath_*\unc{\sh{F}}\xrightarrow{\beta_+}  \pfun{m}{\Xi}\sh{F}\xrightarrow{\alpha_+} \jmath_*\sh{F}\to 0
\end{equation*}
where the composition $\alpha_+\alpha_-\colon\jmath_!\sh{F}\to\jmath_*\sh{F}$ is the natural morphism and we have $\beta_-\beta_+=1-t\colon\imath_*\unc{\sh{F}}\to \imath_*\unc{\sh{F}}$. 

The above definition of Beilinson's maximal extension implies automatically that $\pfun{m}{\Xi}$ is a functor and that it commutes with Verdier duality. We refer to \cite[pag.112]{MR2671769} for a detailed explanation on the maximality of Beilinson's extension.

\section{Extensions over a closed stratum}\label{section:small_extensions}

Let $X$ be a topologically stratified space and $S\subset X$ a closed stratum. We denote by $\imath\colon S\to X$ the closed inclusion of $S$ into $X$ and by $\jmath\colon U=X\smallsetminus S\to X$ the complementary open inclusion. In the rest of the paper $p$ is any perversity on $X$.

\subsection{Extensions over a closed stratum}\label{extf} Let us fix a perverse sheaf $\sh{F}\in\perv{p}{U}$. In order to study its extensions over the closed stratum $S$, we consider the category $\extf{S}{\sh{F}}$ of {\em extensions of $\sh{F}$} whose
\begin{itemize}
\item objects are pairs $(\sh{E},\phi)$, where $\sh{E}\in\perv{p}{X}$ and $\phi\colon\jmath^*\sh{E}\to \sh{F}$ is an isomorphism in $\perv{p}{U}$;
\item morphisms are maps $\epsilon\colon\sh{E}\to\sh{E}'$ making the diagram
\[
\begin{tikzcd}
\jmath^*\sh{E} \ar[rr,"\jmath^*\epsilon"]\ar[dr,swap,"\phi"] && \jmath^*\sh{E}'\ar[dl,"\phi'"]\\
& \sh{F}
\end{tikzcd}
\]
commute.
\end{itemize}

Recall that an {\em initial object} (respectively {\em terminal object}) of a category $\cat{C}$ is an object $\sh{I}\in\cat{C}$ (respectively $\sh{T}\in\cat{C}$) such there exists a unique morphism $\sh{I}\to\sh{C}$ (respectively $\sh{C}\to\sh{T}$) for any object $\sh{C}\in\cat{C}$. The adjunctions in \cref{equation:6ffab} give the following characterisation of two particular extensions of $\sh{F}\in\perv{p}{U}$.

\begin{lemma}
Let $\sh{F}$ be a fixed perverse sheaf in $\perv{p}{U}$. Then, $\pfun{p}{\jmath_!}\sh{F}\in\perv{p}{X}$ is initial in $\extf{S}{\sh{F}}$ and $\pfun{p}{\jmath_*}\sh{F}\in\perv{p}{X}$ is final in $\extf{S}{\sh{F}}$.
\end{lemma}

Important properties of morphisms between objects in $\perv{p}{X}$ which are extensions of a fixed perverse sheaf $\sh{F}\in\perv{p}{U}$ are preserved by the restriction functors introduced in \cref{restriction_functors}.

\begin{lemma}\label{preserving mono/epic}
The functors $\pfun{p}{\imath^!}$ and $\pfun{p}{\imath^*}$ preserve monomorphisms and epimorphisms in $\extf{S}{\sh{F}}$.
\end{lemma}
\begin{proof}
The two statements are dual to each other, therefore we only prove the first claim. Let $\sh{G},\sh{H}\in\extf{S}{\sh{F}}$ and consider a morphism $\gamma\in\Mor{\extf{S}{\sh{F}}}{\sh{G}}{\sh{H}}$. We have two triangles
\begin{equation*}
\imath_*\imath^!\sh{G}\to\sh{G}\to\jmath_*\jmath^*\sh{G}\cong\jmath_*\sh{F}\to\imath^!\sh{G}[1]\quad \hbox{and} \quad
\imath_*\imath^!\sh{H}\to\sh{H}\to\jmath_*\jmath^*\sh{H}\cong\jmath_*\sh{F}\to\imath^!\sh{H}[1]
\end{equation*}
in $\constr{c}{X}$. Considering perverse cohomology yields the diagram
\begin{equation*}\begin{tikzcd} 
0\cong\pcoh{-1}{\jmath_*\sh{F}}\ar[d,equal]\ar[r] & \pfun{p}{\imath^!}\sh{G}\ar[d,swap,"\pfun{p}{\imath^!}\gamma"]\ar[r,hook,"\imath_{\sh{G}}"] & \sh{G}\ar[r]\ar[d,"\gamma"] & \pfun{p}{\jmath_*}\sh{F} \ar[d,equal]\ar[r]& \pcoh{1}{\imath^!\sh{G}}\ar[r]\ar[d] &\ldots\\
0\cong\pcoh{-1}{\jmath_*\sh{F}}\ar[r] & \pfun{p}{\imath^!}\sh{H}\ar[r,hook,"\imath_{\sh{H}}"]  & \sh{H}\ar[r] & \pfun{p}{\jmath_*}\sh{F}\ar[r] & \pcoh{1}{\imath^!\sh{H}}\ar[r] &\ldots
\end{tikzcd}\end{equation*}

If $\gamma$ is monic, then so is $\gamma i_{\sh{G}}=i_{\sh{H}} \pfun{p}{\imath^!}\gamma$ and hence $\pfun{p}{\imath^!}\gamma$ is monic. If $\gamma$ is epic, by diagram chase, also $\pfun{p}{\imath^!}\gamma$ is epic.
\end{proof}

The following result directly follows from the definition of the intermediate restriction functor in \cref{restriction_functors} and \cref{preserving mono/epic}.

\begin{corollary}\label{intermediate restr preserves}
The intermediate restriction functor $\pfun{p}{\imath^{!*}}$ preserves monomorphisms and epimorphisms in $\extf{S}{\sh{F}}$. 
\end{corollary}

\begin{remark}
Note that in the proof of \cref{preserving mono/epic} it is crucial that one considers objects in $\extf{S}{\sh{F}}$ for a fixed perverse sheaf $\sh{F}\in\perv{p}{U}$ as this implies $\pfun{p}{\jmath_*}\jmath^*\sh{G}\cong\pfun{p}{\jmath_*}\jmath^*\sh{H}\cong\pfun{p}{\jmath_*}\sh{F}$. The situation described in \cref{preserving mono/epic}, and hence in \cref{intermediate restr preserves}, is not necessarily true if one considers morphisms between objects which are extensions of different objects in $\perv{p}{U}$. 
\end{remark}

\subsection{Projection functors}\label{section:ProjfuncPlshPls}
We now define two additive functors
\begin{equation}\label{def proj functors}\begin{split}
&\Plsh:=\im(\pfun{p}{\jmath_!}\jmath^*\to\id)\colon\perv{p}{X}\to\perv{p}{X}\\
&\Pls:=\im(\id\to\pfun{p}{\jmath_*}\jmath^*)\colon\perv{p}{X}\to\perv{p}{X}.
\end{split}\end{equation} 
We will refer to the two functors in \cref{def proj functors} as \defn{projection functors}. We note that if $\sh{E}\in\extf{S}{\sh{F}}$, there is a canonical morphism
\begin{equation}\label{can morphism proj functors}\begin{tikzcd} 
\Plsh\sh{E}\ar[rr,swap,"\beta\alpha"]\ar[dr,two heads,swap,"\alpha"]&&\Pls\sh{E}\\
&\pfun{p}{\jmath_{!*}}\sh{F}\ar[ur,hook,swap,"\beta"]
\end{tikzcd}.\end{equation}

The morphism \cref{can morphism proj functors} will play an important role in the rest.

Given a perverse sheaf $\sh{E}\in\perv{p}{X}$, the objects $\Plsh\sh{E}$ and $\Pls\sh{E}$ appear in some particular short exact sequences. 

\begin{lemma}\label{SESs proj functors}
Let $\sh{E}\in\perv{p}{X}$. Then in $\perv{p}{X}$ there are short exact sequences
\begin{equation*} 
0\to\Plsh\sh{E}\to\sh{E}\to\imath_*\pfun{p}{\imath^*}\sh{E}\to0 \quad \hbox{and} \quad 
0\to\imath_*\pfun{p}{\imath^!}\sh{E}\to\sh{E}\to\Pls\sh{E}\to0.
\end{equation*}
\end{lemma}
\begin{proof}
The two short exact sequence are dual to each other, therefore we only prove the first statement. Let us consider the triangle $\jmath_!\jmath^*\sh{E}\to\sh{E}\to\imath_*\imath^*\sh{E}\to \jmath_!\jmath^*\sh{E}[1]$ in $\constr{c}{X}$. The claim follows by taking perverse cohomology of the above triangle, using that $\jmath_!$ is right t-exact and the definition of the projection functor given in \cref{def proj functors}.
\end{proof}

\cref{SESs proj functors} allows us to give an interpretation of $\Plsh\sh{E},\Pls\sh{E}\in\perv{p}{X}$ in terms of sub-objects and quotients supported on $S$.

\begin{remark}\label{rem:projection_in_kernel}
Let $\sh{E}\in\extf{S}{\sh{F}}$ for a fixed perverse sheaf $\sh{F}\in\perv{p}{X}$. By considering the two short exact sequences of \cref{SESs proj functors} and applying adjunctions \cref{equation:6ffab}, we have $\Plsh\sh{E}\in\ker\pfun{p}{\imath^*}$ and $\Pls\sh{E}\in\ker\pfun{p}{\imath^!}$. Thus, $\Plsh\sh{E}$ is the minimal sub-object of $\sh{E}$ such that the quotient is supported on $S$ and $\Pls\sh{E}$ is the minimal quotient of $\sh{E}$ such that the kernel is supported on $S$.
\end{remark}

We study how the projection functors interact with properties of morphisms in the category of extensions of $\sh{F}\in\perv{p}{U}$.

\begin{lemma}\label{lem:plshplsmonoepiiso}
Let $\sh{F}\in\perv{p}{U}$ be a fixed perverse sheaf and $\gamma\in\Mor{\extf{S}{\sh{F}}}{\sh{G}}{\sh{H}}$. Then $\Plsh\gamma$ is an epimorphism and $\Pls\gamma$ is a monomorphism. 
\end{lemma}
\begin{proof}
The two statements are dual to each other, therefore we only prove only the first one. Let us consider the diagram 
\[
\begin{tikzcd} [row sep={0.6cm}]
\sh{G} \ar[rrr,hook,"\gamma"] &&& \sh{H} \\
& \Plsh\sh{G} \ar[r,"\Plsh\gamma"] \ar[ul,hook,"i_{\sh{G}}"] & \Plsh\sh{H} \ar[ur,hook,"i_{\sh{H}}"]& \\ 
\pfun{p}{\jmath_!}\sh{G} \ar[rrr,swap,"\jmath^*\gamma=\id"] \ar[uu] \ar[ur,two heads, "{\pi_{\sh{G}}}"]&&& \pfun{p}{\jmath_!}\sh{G} \ar[uu] \ar[ul,two heads, "{\pi_{\sh{H}}}"]
\end{tikzcd}
\]
We have that $\Plsh\gamma  \pi_{\sh{G}}=\pi_{\sh{H}} \jmath^*\gamma= \pi_{\sh{H}}$
is an epimorphism, then so is $\Plsh\gamma$. 
\end{proof}

\begin{remark}\label{proj_fun_isos}
In the setting of \cref{lem:plshplsmonoepiiso}, if $\gamma$ is assumed to be a monomorphism, then $\Plsh\gamma$ is an isomorphism. This follows since under this further assumption $\gamma i_{\sh{G}}=i_{\sh{H}}\Plsh\gamma$ is a monomorphism and so is $\Plsh\gamma$. Similarly, if $\gamma$ is an epimorphism, then $\Pls\gamma$ is an isomorphism.

We also note that the converse of the above result does not hold in general. For instance, consider the projection $\pi:\sh{G}\oplus\sh{H}\to\sh{G}$ in $\perv{p}{X}$, where $\sh{H}\in\perv{p}{S}$ is a non-zero object. Then $\Plsh\pi$ is an isomorphism although $\pi$ is not monic. 
\end{remark}

Let $\sh{E},\sh{G}\in\extf{S}{\sh{F}}$ be extensions of a fixed perverse sheaf $\sh{F}\in\perv{p}{X}$ and $\gamma\in\Mor{\extf{S}{\sh{F}}}{\sh{E}}{\sh{G}}$. By \cref{lem:plshplsmonoepiiso} we have 
\begin{equation*}\Plsh\sh{E}\twoheadrightarrow \Plsh\im\gamma \cong \Plsh\sh{G} \quad \hbox{and} \quad 
\Pls\sh{E}\cong\Pls\im\gamma\hookrightarrow \Pls\sh{G}.
\end{equation*}

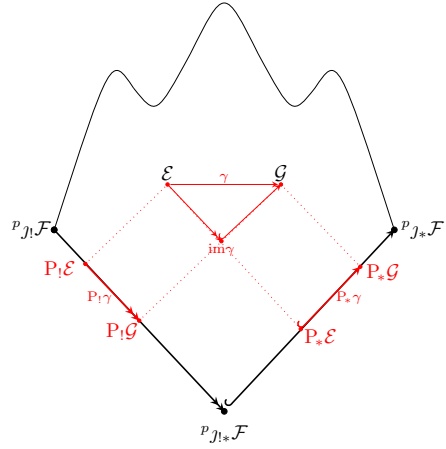
\begin{figure}[H]
\centering
\begin{adjustbox}{scale=0.75, center}
\begin{tikzpicture}
\draw[thick,->>,>=stealth](0,3.2)--(3,0);
\draw[thick,right hook->,>=stealth](3.07,0.07)--(6,3.2);
\draw [fill] (0,3.2) circle [radius=.05];
\draw [fill] (3,0) circle [radius=.05];
\draw [fill] (0,3.2) circle [radius=.05];
\draw [fill] (6,3.2) circle [radius=.05];
\draw [fill,red] (0.55,2.6) circle [radius=.03];
\draw [fill,red] (1.5,1.6) circle [radius=.03];
\draw [fill,red] (5.4,2.55) circle [radius=.03];
\draw [fill,red] (4.35,1.45) circle [radius=.03];
\draw [fill,red] (2,4) circle [radius=.03];
\draw [fill,red] (4,4) circle [radius=.03];
\draw [rounded corners] (0,3.2) .. controls (1,6.3) .. (1.4,5.7);
\draw[rounded corners] (1.4,5.7) .. controls (1.8,5.2) .. (2.4 ,6.3);
\draw[rounded corners] (2.4,6.3) .. controls (3,7.5) .. (3.6,6.3);
\draw [rounded corners] (6,3.2) .. controls (5,6.3) .. (4.6,5.7);
\draw[rounded corners] (4.6,5.7) .. controls (4.2,5.2) .. (3.6,6.3);
\node at (3,-0.4){$\pfun{p}{\jmath_{!*}}\sh{F}$}; 
\node at (-0.4,3.2){$\pfun{p}{\jmath_{!}}\sh{F}$};
\node at (6.5,3.2){$\pfun{p}{\jmath_{*}}\sh{F}$};
\node at (2,4.2){$\sh{E}$};
\node at (4,4.2){$\sh{G}$};
\node at (3,4.1){$\color{red}{\scriptstyle{\gamma}}$};
\draw [fill,red] (2.95,3) circle [radius=.03];
\node at (2.95,2.85){$\color{red}{\scriptstyle{\im\gamma}}$};
\path[->,>=stealth,red](2.05,4) edge (3.95,4);
\path[dotted,red](0.55,2.6) edge (2,4);
\path[dotted,red] (1.5,1.6) edge (4,4);
\path[dotted,red] (2,4) edge (4.35,1.45);
\path[dotted,red] (4,4) edge (5.4,2.55); 
\path[thick,->>,>=stealth,red](0.55,2.6) edge (1.5,1.6);
\path[->>,>=stealth,red](2,4) edge (2.95,3);
\path[thick,right hook->,>=stealth,red] (4.35,1.45) edge (5.4,2.55);
\path[right hook->,>=stealth,red] (2.952,3.02) edge (4,4);
\node at (0.1,2.5) {$\color{red}{\Plsh\sh{E}}$};
\node at (1.2,1.4) {$\color{red}{\Plsh\sh{G}}$};
\node at (4.7,1.3) {$\color{red}{\Pls\sh{E}}$};
\node at (5.8,2.45) {$\color{red}{\Pls\sh{G}}$};
\node at (0.8,2) {$\color{red}{\scriptstyle{\Plsh\gamma}}$};
\node at (5.2,2) {$\color{red}{\scriptstyle{\Pls\gamma}}$};
\end{tikzpicture}
\end{adjustbox}
\caption{Behaviour of the projection functors on a morphism $\gamma\in\Mor{\extf{S}{\sh{F}}}{\sh{E}}{\sh{G}}$.} \label{fig:ME2}
\end{figure}

\subsection{Small extensions}\label{subsec_small_exts}
We will be particularly interested in the class of objects $\sh{E}\in\perv{p}{X}$ which are sent to zero by the intermediate restriction functor introduced in \cref{restriction_functors}. 

\begin{definition}
An object $\sh{E}\in\perv{p}{X}$ is \defn{small with respect to $S$} (or \defn{$S$-small}) if $\pfun{p}{\imath^{!*}}\sh{E}=0$.
\end{definition}

We will denote the \defn{category of $S$-small extensions over $S$} by $\sextf{S}{\sh{F}}$. 

It is easy to note that the extensions introduced in \cref{extension_functors} and \cref{Beilinson_max_ext} are examples of $S$-small objects.

\begin{remark} Let $\sh{F}\in\perv{p}{U}$. The objects $\pfun{p}{\jmath_!}\sh{F},\pfun{p}{\jmath_*}\sh{F}\in\perv{p}{X}$ are $S$-small in $\extf{S}{\sh{F}}$ by adjunction. The intermediate extension $\pfun{p}{\jmath_{!*}}\sh{F}$ is also $S$-small in $\extf{S}{\sh{F}}$ as it has no non-zero sub-object or quotient supported on $S$, see \cref{extension_functors}. Beilinson's maximal extension $\pfun{m}{\Xi}\sh{F}$ of $\sh{F}\in\perv{m}{X}$ is $S$-small in $\extf{S}{\sh{F}}$ by \cite[pag.112]{MR2671769} when it exists - see \cref{Beilinson_max_ext}. 
\end{remark}

From now on, we will assume that the category $\perv{p}{S}=\loc{S}[-p(S)]$ is semisimple. This will allow us to give a topological interpretation of $S$-small perverse sheaves in terms of summands supported on the closed stratum $S\subset X$.

\begin{remark}
By Maschke's Theorem, the assumption on the semisimplicity of $\loc{S}$ holds if $\pi_1(S)$ is finite and one works with a field whose characteristic does not divide the order of $\pi_1(S)$.
\end{remark}

\begin{lemma}\label{small_no_summ_S}
A perverse sheaf $\sh{E}\in\perv{p}{X}$ is $S$-small if and only if it has no summand supported on $S$. 
\end{lemma}
\begin{proof}
Let us suppose $\sh{E}\in\perv{p}{X}$ $S$-small, that is $\pfun{p}{\imath^{!*}}\sh{E}\cong0$. By \cref{max_sub_quot_on_Z} $\imath_*\pfun{p}{\imath^!}\sh{E}$ and $\imath_*\pfun{p}{\imath^*}\sh{E}$ are the maximal quotient and maximal sub-object supported on $S$ respectively. By considering the diagram
\[
\begin{tikzcd}
		&\sh{E}\ar[dr,two heads] &\\
	\imath_*\pfun{p}{\imath^!}\sh{E}\ar[ur,hook] \ar[rr,"0"]&& \imath_*\pfun{p}{\imath^*}\sh{E}
\end{tikzcd}
\]
it follows that $\sh{E}$ cannot have summands in $\perv{p}{S}$. 

Now suppose that $\sh{E}\in\perv{p}{X}$ has no summand supported on $S$. Since $\perv{p}{S}$ is semisimple, the morphism $\pfun{p}{\imath^!}\sh{E} \twoheadrightarrow \pfun{p}{\imath^{!*}}\sh{E}$ splits. Thus there is a monomorphism $\imath_*\pfun{p}{\imath^{!*}}\sh{E}\monic\imath_*\pfun{p}{\imath^!}\sh{E}$ and we have $\sh{E}\cong\imath_*\pfun{p}{\imath^{!*}}\sh{E}\oplus\sh{E}'$. The $S$-smallness of $\sh{E}\in\perv{p}{X}$ follows. 
\end{proof}

We now show that the $S$-smallness property is inherited by sub-objects and quotients in the category $\extf{S}{\sh{F}}$.

\begin{proposition}\label{prop:closed_under_quot_and_subs}
Sub-objects and quotients of $S$-small objects in $\extf{S}{\sh{F}}$ are $S$-small.
\end{proposition}
\begin{proof}
The two statements are dual, hence we only prove the first one. Let $\sh{E}\in\perv{p}{X}$ be an $S$-small extension of a fixed $\sh{F}\in\perv{p}{U}$ and consider a sub-object $\sh{E}'\in\perv{p}{X}$ of $\sh{E}$ in $\extf{S}{\sh{F}}$, that is there is a monomorphism $i_{\sh{E}'}:\sh{E}'\monic\sh{E}$ in $\perv{p}{X}$ such that $\jmath^* i_{\sh{E}'}=\id_{\sh{F}}$. We have a diagram
\begin{equation*}\begin{tikzcd} [row sep={0.6cm}]
\pfun{p}{\imath^!}\sh{E}'\ar[d,hook] \ar[rr] && \pfun{p}{\imath^*}\sh{E}' \ar[d,hook]\\
\pfun{p}{\imath^!}\sh{E}\ar[rr] \ar[dr,two heads] && \pfun{p}{\imath^*}\sh{E} \\
&\pfun{p}{\imath^{!*}}\sh{E}\cong0\ar[ur,hook]
\end{tikzcd}\end{equation*}
where the vertical arrows are monomorphisms by Lemma \ref{preserving mono/epic}. Therefore, the top horizonatal map in the above diagram is zero. Hence $\pfun{p}{\imath^{!*}}\sh{E}'=0$, that is $\sh{E}'\in\perv{p}{X}$ is $S$-small.
\end{proof}

\begin{remark}\label{rem:SESF_not_abelian}
The category $\sextf{S}{\sh{F}}$ is additive, and closed under quotients and sub-objects by \cref{prop:closed_under_quot_and_subs}. In general, it fails to be a Serre subcategory of $\perv{p}{X}$, since the intermediate restriction functor is not exact. Moreover, $\sextf{S}{\sh{F}}$ is not abelian. For instance, in the setting of \cref{ex:P1} the morphism $\gamma\colon\jmath_!\kk_{\C}[1]\to\jmath_*\kk_{\C}[1]$ is such that $\ker\gamma\cong\coker\gamma\cong\imath_*\kk_{\pt}\not\in\sextf{S}{\sh{F}}$ by \cref{small_no_summ_S}.
\end{remark}

\subsection{Extending a perverse sheaf over a closed stratum}\label{subsection:extending_pervsheaves}

Let us fix $\sh{F}\in\perv{p}{U}$. Given an object $\sh{A}\in\extf{S}{\sh{F}}$ with $\pfun{p}{\imath^*}\sh{A}\cong0$, there is a  short exact sequences $0\to\imath_*\pfun{p}{\imath^!}\sh{A}\to\sh{A}\to\pfun{p}{\jmath_{!*}}\sh{F}\to0$ in $\perv{p}{X}$, that is an element
\[
\epsilon_{\sh{A}}\in\Ext{\perv{p}{X}}{1}{\pfun{p}{\jmath_{!*}}\sh{F}}{\imath_*\pfun{p}{\imath^!}\sh{A}} .
\]
Similarly, given a perverse sheaf $\sh{B}\in\perv{p}{X}$ with $\pfun{p}{\imath^!}\sh{B}\cong0$ there is a short exact sequences $0\to\pfun{p}{\jmath_{!*}}\sh{F}\to\sh{B}\to\imath_*\pfun{p}{\imath^*}\sh{B}\to0$ in $\perv{p}{X}$, that is an element
\[
\epsilon_{\sh{B}}\in  \Ext{\perv{p}{X}}{1}{\imath_*\pfun{p}{\imath^*}\sh{B}}{\pfun{p}{\jmath_{!*}}\sh{F}}.
\]

\begin{definition}\label{ext_pair_relative}
The pair $(\sh{A},\sh{B})\in\perv{p}{X}^2$ is an \defn{extension pair relative to} $\sh{F}\in\perv{p}{U}$ provided that:
\begin{enumerate}
\item $\sh{A},\sh{B}\in\extf{S}{\sh{F}}$.
\item $\pfun{p}{\imath^*}\sh{A}\cong0\cong\pfun{p}{\imath^!}\sh{B}$.
\item The class $\epsilon_{\sh{A}}\epsilon_{\sh{B}}=0$ in $\Mor{\constr{c}{S}}{\pfun{p}{\imath^*}\sh{B}[-1]}{\pfun{p}{\imath^!}\sh{A}[1]}$.
\end{enumerate}
\end{definition}

\begin{remark}\leavevmode
\begin{enumerate}
\item The data of an extension pair $(\sh{A},\sh{B})$ is equivalent to a pair of objects which restrict to $\sh{F}$ on $U$ with no quotient supported and no sub-object supported on $S$ respectively, and two classes $\epsilon_{\sh{A}}$ and $\epsilon_{\sh{B}}$ such that the triangulated class $\epsilon_{\sh{A}}\epsilon_{\sh{B}}$ vanishes.
\item If one assumes that $S$ is a closed stratum of $X$ such that $H^2(S;\sh{L})\cong0$ for any $\sh{L}\in\loc{S}$, then the third condition in \cref{ext_pair_relative} is always satisfied and an extension pair relative to $\sh{F}$ is given just by a pair of perverse sheaves on $X$ which restrict to $\sh{F}$ such that the first one has no quotient supported on $S$ and the second has no sub-object supported on $S$. We remark that although the condition $H^2(S;\sh{L})\cong0$ for any $\sh{L}\in\loc{S}$ simplifies the situation, it will not be needed in the results we obtain. 
For instance, in the (much more restrictive) cases when the category $\perv{p}{X}$ is highest weight, see \cite[Theorem 5.17]{ParshallScott}, or even faithful highest weight, see \cref{perverse_def_funct} and \cite[Theorem 3.4]{CW26}, $X$ is a topologically stratified space with finitely many contractible strata, each with finite fundamental group. Hence, the condition on the vanishing of the second cohomology group of the closed stratum automatically holds.
\end{enumerate}
\end{remark}

Extension pairs relative to a fixed perverse sheaf $\sh{F}\in\perv{p}{X}$ and $S$-small extensions of $\sh{F}$ turn out to be closely related.

\begin{theorem}\label{thm:thmonetoonecorrsmallext}
Let $X$ be a topologically stratified space, $S\subset X$ a closed stratum so that $\loc{S}$ is semisimple and $\sh{F}\in\perv{p}{U}$. 

There is a bijection between (isomorphism classes of) extension pairs $(\sh{A},\sh{B})$ relative to a fixed $\sh{F}\in\perv{p}{X\smallsetminus S}$ and (isomorphism classes of) $S$-small extension $\sh{E}(\sh{A},\sh{B})\in\extf{S}{\sh{F}}$.\end{theorem}
\begin{proof}
Suppose that $(\sh{A},\sh{B})$ is an extension pair relative to a fixed $\sh{F}\in\perv{p}{U}$. Then, there are triangles
\begin{equation}\label{triang ext pairs}
\imath_*\pfun{p}{\imath^!}\sh{A}\to\sh{A}\to\pfun{p}{\jmath_{!*}}\sh{F}\overset{\epsilon_{\sh{A}}}{\to}\imath_*\pfun{p}{\imath^!}\sh{A}[1] \quad \hbox{and} \quad
\imath_*\pfun{p}{\imath^*}\sh{B}[-1]\overset{\epsilon_{\sh{B}}}{\to}\pfun{p}{\jmath_{!*}}\sh{F}\to\sh{B}\to\imath_*\pfun{p}{\imath^*}\sh{B}
\end{equation}
in $\constr{c}{X}$, where $\epsilon_{\sh{A}}\epsilon_{\sh{B}}$ is the zero class. Since the category $\perv{p}{S}$ is semisimple, then we have that  $\Mor{\constr{c}{X}}{\imath_*\pfun{p}{\imath^*}\sh{B}[-1]}{\imath_*\pfun{p}{\imath^!}\sh{A}}\cong0$ and applying the functor $\Mor{\constr{c}{X}}{\imath_*\pfun{p}{\imath^*}\sh{B}[-1]}{-}$ to the first triangle in \cref{triang ext pairs} gives an exact sequence
\begin{equation*}
0\to\Mor{}{\imath_*\pfun{p}{\imath^*}\sh{B}[-1]}{\sh{A}}\to\Mor{}{\imath_*\pfun{p}{\imath^*}\sh{B}[-1]}{\pfun{p}{\jmath_{!*}}\sh{F}}\overset{\cdot\epsilon_{\sh{B}}}{\to}\Mor{}{\imath_*\pfun{p}{\imath^*}\sh{B}[-1]}{\imath_*\pfun{p}{\imath^!}\sh{A}[1]}\to\ldots 
\end{equation*}
Since $\epsilon_{\sh{A}}\epsilon_{\sh{B}}=0$, there is a unique factorisation of $\epsilon_{\sh{A}}$ via $\widetilde{\epsilon_{\sh{A}}}\in\Mor{}{\imath_*\pfun{p}{\imath^*}\sh{B}[-1]}{\sh{A}}$. We define $\sh{E}(\sh{A},\sh{B})=\mbox{cone}(\widetilde{\epsilon_{\sh{A}}})$. By construction $\sh{E}(\sh{A},\sh{B})\in\perv{p}{X}$ is an extension of $\sh{F}$ which, up to isomorphism, only depends on the isomorphism class of $(\sh{A},\sh{B})$. From the short exact sequence 
\begin{equation}\label{eqn:SES_used_octa}
0\to\sh{A}\to\sh{E}(\sh{A},\sh{B})\to\imath_*\pfun{p}{\imath^*}\sh{B}\to0
\end{equation}
in $\perv{p}{X}$, we deduce that $\pfun{p}{\imath^*}\sh{E}(\sh{A},\sh{B})\cong\pfun{p}{\imath^*}\sh{B}$ and hence $\sh{A}\cong\Plsh\sh{E}(\sh{A},\sh{B})$. Applying the octahedral axiom applied to the factorisation of $\epsilon_{\sh{A}}$ gives a diagram 
\begin{equation}\label{octa}\begin{tikzcd} 
\imath_*\pfun{p}{\imath^*}\sh{B}[-1]\ar[rr,"\widetilde{\epsilon_{\sh{A}}}"] \ar[rrrr,bend left=20, "\epsilon_{\sh{A}}"]&& \sh{A} \ar[dl]\ar[rr]&& \pfun{p}{\jmath_{!*}}\sh{F} \ar[dl]\ar[ddll,bend left]\\
& \sh{E}(\sh{A},\sh{B})\ar[dr]\ar[ul,dashed]&& \imath_*\pfun{p}{\imath^!}\sh{A}[1] \ar[ul,dashed]\ar[ll,dashed]&\\
&& \sh{B}\ar[ur]\ar[uull,dashed,bend left] &&
\end{tikzcd}\end{equation}
where the dashed arrows denote the shift by $[1]$, the triangles with odd numbers of dashed arrows are exact and those with even number of dashed arrows are commutative. Thus, we have a short exact sequence
\begin{equation*}
0\to\imath_*\pfun{p}{\imath^!}\sh{A}\to\sh{E}(\sh{A},\sh{B})\to\sh{B}\to0
\end{equation*}
in $\perv{p}{X}$ which implies that $\pfun{p}{\imath^!}\sh{E}(\sh{A},\sh{B})\cong\pfun{p}{\imath^!}\sh{A}$ and hence $\sh{B}\cong\Pls\sh{E}(\sh{A},\sh{B})$. That is, starting with $(\sh{A},\sh{B})$ an extension pair relative to a fixed $\sh{F}\in\perv{p}{U}$ we have a diagram
\begin{equation}\begin{tikzcd}[column sep={0.6cm}, row sep={0.4cm}] \label{comm diag exts} 
& \imath_*\pfun{p}{\imath^!}\sh{A}\ar[rr] \ar[dr] && \imath_*\pfun{p}{\imath^*}\sh{B}\ar[dr,equal] & \\
\imath_*\pfun{p}{\imath^!}\sh{A} \ar[dr] \ar[ur,equal]&& \sh{E}(\sh{A},\sh{B})\ar[dr] \ar[ur]&& \imath_*\pfun{p}{\imath^*}\sh{B}\\
&\sh{A}\ar[ur] \arrow{dr}{\alpha}  && \sh{B}\ar[dr]  \ar[ur]& \\
\imath_*\pfun{p}{\imath^*}\sh{B}[-1]\ar[dr,equal]\ar[ur]&& \pfun{p}{\jmath_{!*}}\sh{F} \arrow{dr}{\epsilon_{\sh{A}}}\arrow{ur}{\beta}  && \imath_*\pfun{p}{\imath^!}\sh{A}[1]\\
&\imath_*\pfun{p}{\imath^*}\sh{B}[-1] \arrow{ur}{\epsilon_{\sh{B}}}\ar[rr,"0"]&& \imath_*\pfun{p}{\imath^!}\sh{A}[1]\ar[ur,equal]&
\end{tikzcd}\end{equation}

By applying the functor $\Mor{}{\imath_*\pfun{p}{\imath^!}\sh{A}}{-}$ to \cref{eqn:SES_used_octa} and using that $\pfun{p}{\imath^!}\sh{E}(\sh{A},\sh{B})\cong\pfun{p}{\imath^!}\sh{A}$, we have that 
\[
0\cong \Mor{}{\pfun{p}{\imath^!}\sh{A}}{\pfun{p}{\imath^*}\sh{B}}\cong \Mor{}{\pfun{p}{\imath^!}\sh{E}(\sh{A},\sh{B})}{\pfun{p}{\imath^*}\sh{E}(\sh{A},\sh{B})},
\]
that is $\sh{E}(\sh{A},\sh{B})$ is an $S$-small extension of $\sh{F}$. 

Now suppose that $\sh{E}\in\perv{p}{X}$ be an $S$-small extension of a fixed perverse sheaf $\sh{F}\in\perv{p}{U}$. By \cref{SESs proj functors}, there is a commutative diagram 
\begin{equation}\begin{tikzcd}[column sep={0.6cm}, row sep={0.4cm}] \label{comm diag exts} 
& \imath_*\pfun{p}{\imath^!}\sh{E}\ar[rr,swap,"0"] \ar[dr] && \imath_*\pfun{p}{\imath^*}\sh{E}\ar[dr,"\cong"] & \\
\imath_*\pfun{p}{\imath^!}\Plsh\sh{E} \ar[dr] \ar[ur,"\cong"]&& \sh{E}\ar[dr] \ar[ur]&& \imath_*\pfun{p}{\imath^*}\Pls\sh{E}\\
&\Plsh\sh{E}\ar[ur] \arrow{dr}{\alpha}  && \Pls\sh{E}\ar[dr]  \ar[ur]& \\
\imath_*\pfun{p}{\imath^*}\sh{E}[-1]\ar[dr,"\cong"]\ar[ur]&& \pfun{p}{\jmath_{!*}}\sh{F} \arrow{dr}{\epsilon_{\sh{A}}}\arrow{ur}{\beta}  && \imath_*\pfun{p}{\imath^!}\sh{E}[1]\\
&\imath_*\pfun{p}{\imath^*}\Pls\sh{E}[-1] \arrow{ur}{\epsilon_{\sh{B}}}\ar[rr]&& \imath_*\pfun{p}{\imath^!}\Plsh\sh{E}[1]\ar[ur,"\cong"]&
\end{tikzcd}\end{equation}
where the diagonals are exact triangles in $\constr{c}{X}$. The $S$-smallness of $\sh{E}$ implies that the top horizontal morphism in \cref{comm diag exts} is zero, and that $\imath_*\pfun{p}{\imath^!}\Plsh\sh{E}\cong\imath_*\pfun{p}{\imath^!}\sh{E}$ and $\imath_*\pfun{p}{\imath^*}\Pls\sh{E}\cong\imath_*\pfun{p}{\imath^*}\sh{E}$. By construction, we have the classes $\epsilon_{\sh{A}}\in\Ext{}{1}{\pfun{p}{\jmath_{!*}}\sh{F}}{\imath_*\pfun{p}{\imath^!}\Plsh\sh{E}}$ and $\epsilon_{\sh{B}}\in\Ext{}{1}{\imath_*\pfun{p}{\imath^*}\Pls\sh{E}}{\pfun{p}{\jmath_{!*}}\sh{F}}$, and applying the exact functor $\jmath^*$ to the induced short exact sequence we get $\Plsh\sh{E},\Pls\sh{E}\in\extf{S}{\sh{F}}$. Moreover, by \cref{rem:projection_in_kernel}, we have $\pfun{p}{\imath^*}\Plsh\sh{E}\cong\pfun{p}{\imath^!}\Pls\sh{E}\cong0$. Finally, the $S$-small object $\sh{E}$ admits a filtration $\sh{E}\supset \Plsh\sh{E}\supset \imath_*\pfun{p}{\imath^!}\Plsh\sh{E}\supset0$ such that
\[
\sh{E}/\Plsh\sh{E}\cong\imath_*\pfun{p}{\imath^*}\sh{E}\cong\imath_*\pfun{p}{\imath^*}\Pls\sh{E} \quad \hbox{and} \quad \Plsh/\imath_*\pfun{p}{\imath^!}\Plsh\sh{E}\cong\pfun{p}{\jmath_{!*}}\sh{F} 
\]
with induced classes $\Ext{}{1}{\pfun{p}{\jmath_{!*}}\sh{F}}{\imath_*\pfun{p}{\imath^!}\Plsh\sh{E}}$ and $\Ext{}{1}{\imath_*\pfun{p}{\imath^*}\sh{E}}{\Plsh\sh{E}/\imath_*\pfun{p}{\imath^!}\Plsh\sh{E}}$ isomorphic to $\epsilon_{\sh{A}}$ and $\epsilon_{\sh{B}}$ respectively. Thus the class $\epsilon_{\sh{A}}\epsilon_{\sh{B}}\in\Ext{\perv{p}{X}}{2}{\imath_*\pfun{p}{\imath^*}\Pls\sh{E}}{\imath_*\pfun{p}{\imath^!}\Plsh\sh{E}} \cong \Ext{\perv{p}{X}}{2}{\imath_*\pfun{p}{\imath^*}\sh{E}}{\imath_*\pfun{p}{\imath^!}\sh{E}}$ is the zero class. This implies that $\epsilon_{\sh{A}}\epsilon_{\sh{B}}\in\Ext{\constr{c}{X}}{2}{\imath_*\pfun{p}{\imath^*}\sh{E}}{\imath_*\pfun{p}{\imath^!}\sh{E}}$ is the zero class as by definition it is obtained as a product of two (abelian) $\mathrm{Ext}^1$-classes. Thus the bottom horizontal morphism in \cref{comm diag exts} is zero and therefore $(\Plsh\sh{E},\Pls\sh{E})$ is an extension pair relative to $\sh{F}\in\perv{p}{U}$.
\end{proof}

We now introduce a generalised version of Beilinson's maximal extension, see \cref{Beilinson_max_ext} under the only assumption that $S\subset X$ is a closed stratum such that the category $\loc{S}$ is semisimple.
\begin{definition}\label{defn:generalised_max_ext}
	The \emph{maximal extension of a fixed $\sh{F}\in\perv{p}{U}$} is the (isomorphism class of the) perverse sheaf $\pfun{p}{\Xi}\sh{F}\in\perv{p}{X}$ corresponding to the extension pair $(\pfun{p}{\jmath_!}\sh{F},\pfun{p}{\jmath_*}\sh{F})$ relative to a fixed $\sh{F}\in\perv{p}{U}$. 
\end{definition}
 
\begin{remark}\label{rem_on_def__max_ext}\leavevmode
\begin{enumerate}
	\item The maximal extension introduced in \cref{defn:generalised_max_ext} is well-defined by \cref{thm:thmonetoonecorrsmallext}. Since $\constr{c}{X}$ is an algebraic triangulated category, see \cref{constructible_der_cat}, we get a well-defined functor $\pfun{p}{\Xi}\colon\perv{p}{X}\to\perv{p}{X}$. In the setting of \cref{Beilinson_max_ext} it agrees with Beilinson's maximal extension.
	\item In \cite[Proposition 6.1]{MR833195}, the authors prove that when $X$ is either a topologically stratified space or a complex analytic manifolds such that $\pi_1(S)=\pi_2(S)=0$, any perverse sheaf on $X\smallsetminus S$ has a maximal indecomposable extension which is unique (up to non-canonical isomorphism). Since the only hypotheses we use is that $\loc{S}$ is semisimple, \cref{thm:thmonetoonecorrsmallext} extends these results to a wider class of stratified spaces with no restriction on the perversity function.
    \item \cref{thm:thmonetoonecorrsmallext} gives `coordinates' on $S$-small extensions $\sh{E}\in\perv{p}{X}$ of a fixed perverse sheaf $\sh{F}\in\perv{p}{U}$. Indeed, it shows that for any extension pair relative to a perverse sheaf $\sh{F}\in\perv{p}{U}$ there is a unique $S$-small extension $\sh{E}\in\perv{p}{X}$ of $\sh{F}$. This means that in order to understand all $S$-small extensions $\sh{E}$ of $\sh{F}$, it is enough to study objects of the form $\Plsh\sh{G}$ and $\Pls\sh{G}$ for some $\sh{G}\in\perv{p}{X}$, see \cref{fig:M1}.
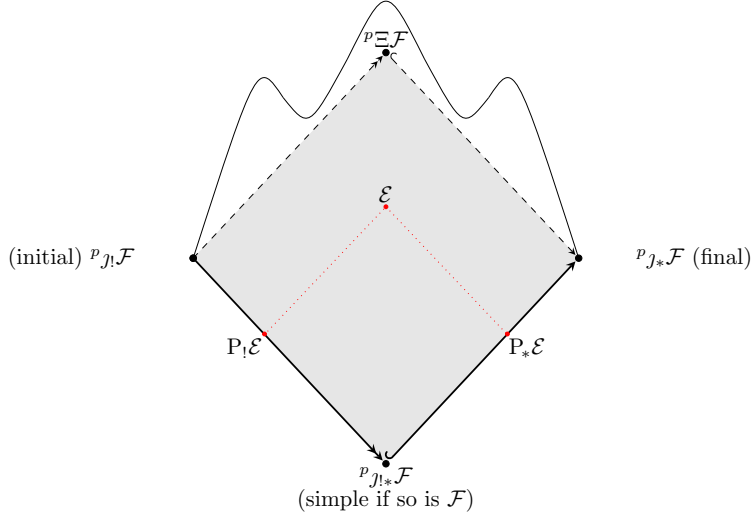
\begin{figure}[H]
\centering
\begin{adjustbox}{scale=0.85, center}
\begin{tikzpicture}
\path[fill=gray!20] (3,0) -- (0,3.2) -- (3,6.4) -- (6,3.2) -- cycle;
\draw[thick,->>,>=stealth](0,3.2)--(2.95,0.05);
\draw[thick,right hook->,>=stealth](3.05,0.05)--(5.95,3.15);
\draw[dashed,->>,>=stealth](0,3.2)--(2.95,6.35);
\draw[dashed,right hook->,>=stealth](3.05,6.35)--(5.95,3.25);
\draw[dotted,red](1.1,2)--(3,4);
\draw[dotted,red](3,4)--(4.9,2);
\draw [fill] (0,3.2) circle [radius=.05];
\draw [fill] (3,0) circle [radius=.05];
\draw [fill] (0,3.2) circle [radius=.05];
\draw [fill] (6,3.2) circle [radius=.05];
\draw [fill] (3,6.4) circle [radius=.05];
\draw [fill,red] (1.11,2.02) circle [radius=.03];
\draw [fill,red] (4.89,2.02) circle [radius=.03];
\draw [fill,red] (3,4) circle [radius=.03];
\draw [rounded corners] (0,3.2) .. controls (1,6.3) .. (1.4,5.7);
\draw[rounded corners] (1.4,5.7) .. controls (1.8,5.2) .. (2.4 ,6.3);
\draw[rounded corners] (2.4,6.3) .. controls (3,7.5) .. (3.6,6.3);
\draw [rounded corners] (6,3.2) .. controls (5,6.3) .. (4.6,5.7);
\draw[rounded corners] (4.6,5.7) .. controls (4.2,5.2) .. (3.6,6.3);
\node at (3,6.6){${\pfun{p}{\Xi}\sh{F}}$};
\node at (3,-0.4){${\substack{\mathlarger{\pfun{p}{\jmath_{!*}}{\sh{F}}} \\ \scriptstyle{\mbox{(simple if so is $\sh{F}$)}}}}$};
\node at (-1.9,3.2){$\mbox{(initial)} \ \pfun{p}{\jmath_!}\sh{F}$};
\node at (7.8,3.2){$\pfun{p}{\jmath_*}\sh{F} \ \mbox{(final)}$};
\node at (3,4.2){$\sh{E}$};
\node at (0.8,1.8){$\Plsh\sh{E}$};
\node at (5.2,1.8){$\Pls\sh{E}$};  
\end{tikzpicture}
\end{adjustbox}
\medskip
\caption{Extensions $\sh{E}\in\perv{p}{X}$ of a fixed perverse sheaf $\sh{F}\in\perv{p}{U}$. The shaded part indicates the $S$-small extensions of $\sh{F}\in\perv{p}{U}$. Objects which lie on the lower left (respectively right) edge have no quotient (respectively sub-object) supported on $S$.} \label{fig:M1}
\end{figure}
\end{enumerate}
\end{remark}

\subsection{The category of extension pairs over a closed stratum}\label{cat_of_ext pairs} Let us consider the {\em category $\ep{X\smallsetminus S}$ of extension pairs over $S$}  whose:
\begin{itemize}
\item objects are extension pairs $(\sh{A},\sh{B})\in\perv{p}{X}^2$ relative to any $\sh{F}\in\perv{p}{X\smallsetminus S}$.
\item morphisms $f=(f_{\sh{A}},f_{\sh{B}})\colon (\sh{A},\sh{B})\to (\sh{A}',\sh{B}')$ such that $\alpha=\alpha' f_{\sh{A}}$ and $\beta'=f_{\sh{B}}\beta$, that is in the diagram
\begin{equation*}\begin{tikzcd}
\sh{A} \ar[r,"\alpha"] \ar[d, swap, "{f_{\sh{A}}}"]& \pfun{p}{\jmath_{!*}}\sh{F}\ar[d,"\id"] \ar[r,"\beta"] & \sh{B} \ar[d, "{f_{\sh{B}}}"] \\
\sh{A}' \ar[r,"\alpha'"] & \pfun{p}{\jmath_{!*}}\sh{F} \ar[r,"\beta'"] & \sh{B}' 
\end{tikzcd}\end{equation*}
both squares commute.
\end{itemize}

\begin{remark}
The category $\ep{X\smallsetminus S}$ is additive, but in general is not abelian. For instance, in the setting of \cref{ex:P1}, the morphism $f=(f_1,\id)\colon(\jmath_!\kk_{\C}[1], \kk_{\PP^1}[1])\to(\jmath_*\kk_{\C}[1],\kk_{\PP^1}[1])$ admits neither a kernel nor a cokernel. 
\end{remark}

\begin{theorem}\label{thm:equivalence_of_cat}
Let $X$ be a topologically stratified space, $S\subset X$ be a closed stratum such that $\loc{S}$ is semisimple and $\sh{F}\in\perv{p}{U}$.

The assignment 
\begin{equation*}\begin{split}
\Theta \colon    \sextf{S}{\sh{F}} &\to \ep{X\smallsetminus S}\\ 
  \sh{E} & \mapsto (\Plsh\sh{E},\Pls\sh{E})
\end{split}\end{equation*}
defines an equivalence of additive categories.
\end{theorem}
\begin{proof}
Since $\constr{c}{X}$ is an algebraic triangulated category, see \cref{constructible_der_cat}, $\Theta$ is an additive essentially surjective functor by the proof of \cref{thm:thmonetoonecorrsmallext}.

Let $\epsilon\in\Mor{\sextf{S}{\sh{F}}}{\sh{E}}{\sh{E}'}$ be such that $\Theta(\epsilon)=\left(\Plsh\epsilon,\Pls\epsilon\right)=0$ and consider the factorisation
\begin{equation*}
\begin{tikzcd}
\sh{E} \ar[rr,"\epsilon"] \ar[dr,two heads] && \sh{E}' \\
&\im\epsilon \ar[ur,hook]
\end{tikzcd}.
\end{equation*}
By \cref{prop:closed_under_quot_and_subs}, we have that $\im\epsilon$ is $S$-small. By \cref{lem:plshplsmonoepiiso} the morphism $\Pls\sh{E}\cong\Pls\im\epsilon \to\Pls\sh{E}'$ is monic and is zero by assumption. Then $\Pls\im\epsilon\cong 0$ and \cref{SESs proj functors} implies that $\im\epsilon\cong\imath_*\pfun{p}{\imath^*}\im\epsilon$. By \cref{small_no_summ_S} the $S$-smallness of $\im\epsilon$ forces $\im\epsilon\cong 0$, thus $\epsilon\cong0$. Therefore, $\Theta$ is faithful.

Let $f=(f_{\sh{A}},f_{\sh{B}})\in\Mor{\ep{X\smallsetminus S}}{(\sh{A},\sh{B})}{(\sh{A}',\sh{B}')}$ and consider the diagram 
\begin{equation*}\label{big_diagr}
\begin{tikzcd}
\imath_*\pfun{p}{\imath^*}\sh{B}[-1] \arrow[ddr, swap, phantom, "(1)" below]  \ar[dr,dashed,"{\exists!\widetilde{\epsilon_{\sh{A}}}}"] \ar[ddd,swap,"{\imath_*\pfun{p}{\imath^*}f_{\sh{B}}}"]\ar[rr, "\epsilon_{\sh{A}}"] && \pfun{p}{\jmath_{!*}}\sh{F} \ar[ddd,"\id"] \ar[r,"\beta"] \arrow[dddr, phantom, "(3)"]& \sh{B} \ar[ddd, "f_{\sh{B}}"]\\
& \sh{A} \ar[d,"f_{\sh{A}}"] \ar[ur,"\alpha'"] \arrow[ddr, phantom,swap, "(2)" above]&& \\
& \sh{A}' \ar[dr,"\alpha"]&& \\
\imath_*\pfun{p}{\imath^*}\sh{B}'[-1] \ar[ur,dashed,"{\exists!\widetilde{\epsilon_{\sh{A}'}}}"]\ar[rr, "\epsilon_{\sh{A}'}"] && \pfun{p}{\jmath_{!*}}\sh{F} \ar[r,"\beta'"]& \sh{B}'
\end{tikzcd},
\end{equation*}
where the subdiagrams $(2)$ and $(3)$ commute by definition and $\beta\epsilon_{\sh{A}}=\beta'\epsilon_{\sh{A}'}=0$. The fact that $\exists!\widetilde{\epsilon_{\sh{A}}}\in\Mor{}{\imath_*\pfun{p}{\imath^*}\sh{B}[-1]}{\sh{A}}$ and $\exists!\widetilde{\epsilon_{\sh{A}'}}\in\Mor{}{\imath_*\pfun{p}{\imath^*}\sh{B}'[-1]}{\sh{A}'}$ is obtained by applying the functor $\Mor{}{\imath_*\pfun{p}{\imath^*}\sh{B}[-1]}{-}$, respectively $\Mor{}{\imath_*\pfun{p}{\imath^*}\sh{B}'[-1]}{-}$, to the second short exact sequence of \cref{SESs proj functors} for $\sh{A}$, respectively for $\sh{A}'$, and using that $\perv{p}{S}$ is semisimple. Similarly, one obtains that $\exists!\mu\colon\imath_*\pfun{p}{\imath^*}\sh{B}[-1]\to\sh{A}'$, hence the sub-diagram $(1)$ also commutes. The unique morphism $\mu\in\Mor{}{\imath_*\pfun{p}{\imath^*}\sh{B}[-1]}{\sh{A}'}$ lifts $\epsilon_{\sh{A}}=\epsilon_{\sh{A}'}\imath_*\pfun{p}{\imath^*}f_{\sh{B}}$ along $\alpha'$, where the latter identity follows since $(3)$, and hence $(1)+(2)$, commutes. The existence of a unique $\mu$ implies that in the diagram of triangles
\begin{equation*}
\begin{tikzcd}
\imath_*\pfun{p}{\imath^*}\sh{B}[-1] \ar[r,"{\widetilde{\epsilon_{\sh{A}}}}"] \ar[d,swap, "{\imath_*\pfun{p}{\imath^*}f_{\sh{B}}[-1]}"]  &\sh{A} \ar[r] \ar[d,swap, "f_{\sh{A}}"] & \sh{E}(\sh{A},\sh{B}) \ar[r] \ar[d,dashed]& \imath_*\pfun{p}{\imath^*}\sh{B} \ar[d,"\imath_*\pfun{p}{\imath^*}f_{\sh{B}}"] \\
\imath_*\pfun{p}{\imath^*}\sh{B}'[-1]  \ar[r,"{\widetilde{\epsilon_{\sh{A}'}}}"] &\sh{A}' \ar[r] & \sh{E}(\sh{A}',\sh{B}') \ar[r] & \imath_*\pfun{p}{\imath^*}\sh{B}' 
\end{tikzcd} 
\end{equation*}
the left hand side square commutes. We define $\epsilon\in\Mor{\sextf{S}{\sh{F}}}{\sh{E}(\sh{A},\sh{B})}{\sh{E}(\sh{A}',\sh{B}')}$ to be the dashed morphism in above diagram. Since the construction of $\sh{E}(\sh{A},\sh{B})$ in \cref{thm:thmonetoonecorrsmallext} is functorial, then $\epsilon\colon\sh{E}\to\sh{E}$ is unique (up to isomorphism). With the diagram \cref{comm diag exts} we verify that $\Theta(\epsilon)=(\Plsh\epsilon,\Pls\epsilon)\cong(f_{\sh{A}},f_{\sh{B}})$, that is $\Theta$ is full.

It follows that $\Theta$ is an equivalence of additive categories.
\end{proof}

Note that an object in $(\sh{A},\sh{B})\in\ep{X\smallsetminus S}$ is indecomposable if and only if the canonical morphism $\beta\alpha\colon\sh{A}\to\sh{B}$ cannot be written as a direct sum of other morphisms. As immediate consequence of \cref{thm:equivalence_of_cat} we obtain a characterisation of indecomposable $S$-small extensions over $S$. 

\begin{corollary}\label{cor_on_indecomp}
Let $X$ be a topologically stratified space, $S\subset X$ a closed stratum such that the category $\loc{S}$ is semisimple and $\sh{F}\in\perv{p}{U}$.

An $S$-small perverse sheaf $\sh{E}\in\sextf{S}{\sh{F}}$ is indecomposable if and only if the canonical morphism $\beta\alpha\colon\Plsh\sh{E}\to\Pls\sh{E}$ cannot be written as a direct sum of morphisms.
\end{corollary}

We now give a characterisation of indecomposable perverse sheaves on $X$ based on the equivalence of categories proved in \cref{thm:equivalence_of_cat} and \cref{cor_on_indecomp}.

\begin{theorem}\label{thm:on:indec_exts}
Let $X$ be a topologically stratified space, $S\subset X$ a closed stratum such that the category $\loc{S}$ is semisimple and $\sh{F}\in\perv{p}{U}$.

Isomorphism classes of indecomposable $p$-perverse sheaves on $X$ are classified by extensions by zero of indecomposable local systems on $S$ and extension pairs relative to $\sh{F}\in\perv{p}{U}$ whose canonical morphism cannot be written a direct sum.
\end{theorem}

\begin{remark}\label{rmk_non_splitting} Note that in general we do not require an extension pair $(\sh{A},\sh{B})\in\ep{X\smallsetminus S}$ to be relative to an indecomposable $\sh{F}\in\perv{p}{U}$. That is if $\sh{F}$ is indecomposable, then the corresponding $S$-small extension is indecomposable as the canonical morphism $\beta\alpha\colon\sh{A}\to \sh{B}$ cannot be written as a direct sum of morphisms. The converse, however, does not hold.
\end{remark}

\section{Examples}\label{sec_on_examples} In this final section we present illustrative examples of some two strata spaces. In the first three examples the category of perverse sheaves is equivalent to a module category over a finite-dimensional algebra while in the last case is not, see \cite[Corollary 5.2]{CIPRIANI_2021}. 

We will denote by $\kk_Y$ the constant sheaf on $Y$.

\begin{example}\label{ex:P1} Let us consider the complex projective line stratified by a point and its complement, that is $X=\C\PP^1 \supset \pt$. We have an open stratum $U\cong\C$ and a closed stratum $S\cong\pt$. Let $m$ be the middle perversity on $X$, that is $(m(U),m(S))=(-1,0)$. We have $\perv{m}{U}\cong\kvect[1]$ and one can check that, up to isomorphism, all the extension pairs relative to the only indecomposable perverse sheaf $\sh{F}\cong \kk_U[1]$ are the following four.
\medskip
\begin{center}  
  \begin{tabular}{|c|c|}
    \hline
      Extension pair & Small extension \\
    \hline
       $(\pfun{m}{\jmath_!}\sh{F},\pfun{m}{\jmath_{!*}}\sh{F}) \cong (\jmath_!\kk_U[1],\kk_X[1])$ 
       &  $\pfun{m}{\jmath_!}\sh{F} \cong \jmath_!\kk_U[1]$     \\ \hline
       $(\pfun{m}{\jmath_{!*}}\sh{F},\pfun{m}{\jmath_{!*}}\sh{F}) \cong (\kk_X[1],\kk_X[1])$ 
       &  $\pfun{m}{\jmath_{!*}}\sh{F} \cong \kk_X[1]$     \\ \hline
       $(\pfun{m}{\jmath_{!*}}\sh{F},\pfun{m}{\jmath_*}\sh{F}) \cong (\kk_X[1],\jmath_*\kk_U[1])$ 
       &  $\pfun{m}{\jmath_*}\sh{F} \cong \jmath_*\kk_U[1]$     \\ \hline
       $(\pfun{m}{\jmath_!}\sh{F},\pfun{m}{\jmath_*}\sh{F}) \cong (\kk_X[1],\jmath_*\kk_U[1])$ 
       &  $\pfun{m}{\Xi}\sh{F}$     \\ \hline
    \end{tabular}
\end{center}
The above four extensions and $\imath_*\kk_{\pt}$ are all the indecomposable objects in $\perv{m}{X}$.
\end{example}

The following example is due to Jon Woolf.

\begin{example}\label{example:kronecker} Let us consider the circle stratified by a point and its complement, that is $X=S^1\supset\pt$. We have two contractible strata $U\cong (0,1)$ and $S=\pt$. Let $o$ be the zero perversity on $X$, that is $(o(U),o(S))=(0,0)$.  Let $\sh{L}\in\loc{U}$ a local system with stalk $V$, then $\jmath_!\sh{L}\in\perv{o}{X}$ is simple. The extension pairs relative to $\sh{L}$ are of the form $(\jmath_!V,\sh{B})$, where $\jmath_!V\hookrightarrow\sh{B}\hookrightarrow \jmath_*V$. The sub-objects $\sh{B}$ of $\jmath_*V$ are in bijection with the sub-objects $\sh{B}'\hookrightarrow \jmath_*V/\jmath_!V\cong V^2$ and hence also with subspaces $i \colon W\to V^2$. The indecomposable representations of the Kronecker quiver $W\rightrightarrows V$ are then in bijection with the vector spaces $W$ and $V$ and maps $a,b\colon W\to V$ given by composing $i$ with the two projections such that $\ker a\cap \ker b=0$. The latter requirement corresponds to the $S$-smallness condition. 
\end{example}

\begin{example}\label{p2_strat_p1} Let us consider the complex projective sphere $X=\C\PP^2$ stratified by a complex projective line and its complement, that is there are two strata $U\cong\C^2$ and $S=\C\PP^1$. Let $m$ be the middle perversity on $X$, that is $(m(U),m(S))=(-2,-1)$. We have $\perv{m}{S}\cong\loc{\C\PP^1}[1]$ and one can check that, up to isomorphism, all the extension pairs relative to $\sh{F}\cong\kk_U[2]$ are 
\medskip
\begin{center}  
  \begin{tabular}{|c|c|}
    \hline
      Extension pair & Small extension \\
    \hline
       $(\pfun{m}{\jmath_!}\sh{F},\pfun{m}{\jmath_{!*}}\sh{F}) \cong (\jmath_!\kk_U[2],\kk_X[2])$ 
       &  $\pfun{m}{\jmath_!}\sh{F} \cong \jmath_!\kk_U[2]$     \\ \hline
       $(\pfun{m}{\jmath_{!*}}\sh{F},\pfun{m}{\jmath_{!*}}\sh{F}) \cong (\kk_X[2],\kk_X[2])$ 
       &  $\pfun{m}{\jmath_{!*}}\sh{F} \cong \kk_X[2]$     \\ \hline
       $(\pfun{m}{\jmath_{!*}}\sh{F},\pfun{m}{\jmath_*}\sh{F}) \cong (\kk_X[2],\jmath_*\kk_U[2])$ 
       &  $\pfun{m}{\jmath_*}\sh{F} \cong \jmath_*\kk_U[2]$     \\ \hline
    \end{tabular}
\end{center}
\medskip
Note that $(\pfun{m}{\jmath_{!}}\sh{F},\pfun{m}{\jmath_{*}}\sh{F})$ is not an extension pair relative to $\sh{F}$ since the third condition in \cref{ext_pair_relative} is not satisfied. This reflects the impossibility noted in \cite[Example 6.3]{MR833195} of defining a unique maximal extension of $\sh{F}\in\perv{p}{U}$. The above three extensions and $\imath_*\kk_{\C\PP^1}[1]$ are all the indecomposable extensions of $\perv{m}{\C\PP^2}$.
\end{example} 

\begin{example}
 Let us consider $X=\C$ stratified by the origin $S=(0,0)$ and its complement $U=\C^*$. Let $m$ be the middle perversity on $X$, that is $(m(U), m(S))=(-1,0)$, and let $\sh{L}[1]\in\perv{m}{U}=\loc{U}[1]$ a (shifted) local system on $U$. We have two distinct cases:
 \begin{enumerate}
 \item If $\sh{L}\not\cong\kk_U$, then $\sh{E}\cong\pfun{m}{\jmath_!}\sh{L}[1]=\pfun{m}{\jmath_*}\sh{L}[1]$ and hence $\sh{E}\in\perv{m}{X}$ is the unique $S$-small extension of $\sh{L}[1]$.
 \item If $\sh{L}\cong	\kk_{U}$, then there are four $S$-small extensions, namely $\kk_X[1],\pfun{m}{\jmath_!}\kk_U[1], \pfun{m}{\jmath_*}\kk_U[1]$ and the maximal extension $\pfun{m}{\Xi}\kk_U[1]$ relative to the pair $(\pfun{m}{\jmath_!}\sh{L}[1], \pfun{m}{\jmath_*}\sh{L}[1])$.
  \end{enumerate} 
  \end{example}

\bibliographystyle{myalpha}
\bibliography{References-2.bib}
\end{document}